\documentclass[a4paper,10pt]{amsart}

\usepackage[english]{babel}
\usepackage{amssymb, amsmath, amsthm}
\usepackage[latin1]{inputenc}
\usepackage{color}
\usepackage{graphicx}
\usepackage[final]{hyperref}
\usepackage[a4paper, centering]{geometry}

\allowdisplaybreaks 

\def\one{\mathbf{1}}

\def\mesnu{\overline{\nu}}

\newcommand{\weakto} {\rightharpoonup}                 
\newcommand{\weakstarto}{\stackrel {*} {\weakto}}      

\newcommand{\mean}[1]{\,-\hskip-1.08em\int_{#1}} 

\newcommand{\pushright}[1]{\ifmeasuring@#1\else\omit\hfill$\displaystyle#1$\fi\ignorespaces} 

\newcommand{\res}{\mathop{\hbox{\vrule height 7pt width .5pt depth 0pt
\vrule height .5pt width 6pt depth 0pt}}\nolimits} 

\DeclareMathOperator{\spt}{spt} 
\DeclareMathOperator{\Div}{div}

\DeclareMathOperator{\sign}{sign}

\newcommand{\M}[1]{\mathcal{#1}}    
\renewcommand{\H}{\M{H}}
\renewcommand{\L}{\M{L}}

\newcommand{\N}{\mathbb{N}}
\newcommand{\R}{\mathbb{R}}

\def\diver{\mathop{\rm div}\nolimits}

\newcommand{\Leb}[1]{\mathcal{L}^{#1}} 
\newcommand{\eps}{\varepsilon}
\newcommand{\vphi}{\varphi}

\newcommand{\W}{W}

\def\uinf{u^{\infty}}
\def\vinf{v^{\infty}}

\def\medint{-\kern  -,375cm\int}

\newcommand{\A}{\boldsymbol A}
\newcommand{\B}{\boldsymbol B}
\newcommand{\F}{\boldsymbol F}
\newcommand{\Abar}{\overline{\boldsymbol A}}

\renewcommand{\a}{\boldsymbol a}

\newtheorem{theorem}{Theorem}[section]

\newtheorem{lemma}[theorem]{Lemma}

\theoremstyle{definition}
\newtheorem{definition}[theorem]{Definition}
\theoremstyle{remark}
\newtheorem{remark}[theorem]{Remark}

\newtheorem*{ack}{Acknowledgements}

\begin{document}

\author[G.~Crasta]{Graziano Crasta}
\address{Dipartimento di Matematica ``G.\ Castelnuovo'', Univ.\ di Roma I\\
P.le A.\ Moro 2 -- I-00185 Roma (Italy)}
\email{crasta@mat.uniroma1.it}
\author[V.~De Cicco]{Virginia De Cicco}
\address{Dipartimento di Scienze di Base  e Applicate per l'Ingegneria, Univ.\ di Roma I\\
Via A.\ Scarpa 10 -- I-00185 Roma (Italy)}
\email{virginia.decicco@sbai.uniroma1.it}
\author[G.~De Philippis]{Guido De Philippis}
\address{Unit\'e de Math\'ematiques Pure et Appliqu\'ees -- ENS de Lyon\\
	UMR 5669 -- UMPA}
\email{guido.de-philippis@ens-lyon.fr}
\author[F.~Ghiraldin]{Francesco Ghiraldin}
\address{Institut f\"ur Mathematik, Universit\"at Z\"urich \\
Winterthurerstrasse 190 -- CH-8057 Z\"urich}
\email{francesco.ghiraldin@math.uzh.ch}

%

\date{September 28, 2015; revised January 19, 2016.}


\title[Multidimensional conservation laws with discontinuous flux]{Structure of solutions of multidimensional
	conservation laws with discontinuous flux
	and applications to uniqueness}

\begin{abstract}
We investigate the structure of solutions of conservation laws with discontinuous flux under quite general assumption on the flux. We show that any entropy solution admits traces on the discontinuity set of the coefficients and we use this to prove the validity of a generalized Kato inequality for any pair of solutions. Applications to uniqueness of solutions are then given.
\end{abstract}

\maketitle

\section{Introduction}
Aim of this paper is to study the structure of solutions of conservation laws with discontinuous flux of the form 
\begin{equation}
\label{f:equa}
\Div_z \A(z,u)=0,
\end{equation}
in order to establish a general framework for studying uniqueness of solutions of the Cauchy problem  associated to the evolutionary equation\footnote{Note that \eqref{f:evo0} is a particular case of \eqref{f:equa} with  \(\A(z,u)=(u,\F(z,u))\) and \(z=(t,x)\).} 
\begin{equation}\label{f:evo0}
u_t + \Div_x \F(t,x,u) = 0,\qquad \text{in}\ (0,+\infty)\times \R^N.
\end{equation}
Here \(\A\) (respectively \(\F\))  is discontinuous in its first variable \(z\)  (respectively \((t,x)\)). More precisely we will assume that \(\A(z,\cdot)\in C^1(\R, \R^n)\), 
\(\A(\cdot, v)\in SBV(\R^n, \R^n) \) where \(SBV\) is the space of special function of bounded variation, see \cite[Chapter 4]{AFP}, and that $\A$ satisfies some mild structural assumptions
listed in Section~\ref{sec:prel}.

\medskip

In recent years, the study of conservation laws with discontinuous flux  
has attracted the attention of many authors since they naturally arise in many models, 
see \cite{AMV,AKR,AnMitr,AP,Cocl,Die,GNPT,disc8,Mitr-DCDS,Pan1} and the references therein.

Even in the case the flux \(\F\) is smooth it is well known that the Cauchy problem associated to~\eqref{f:evo0} it is not well posed and some additional {\em entropy conditions} have to be imposed in order to recover uniqueness of the solution, see \cite{Kruz}. 
In the case of a discontinuous flux, these conditions are still not sufficient to select a unique solution to \eqref{f:evo0} and further dissipation conditions, 
involving the traces of the solutions
on the set 
of discontinuities of the flux, must be imposed
in order to ensure uniqueness.

The problem of existence and uniqueness for solutions of \eqref{f:evo0} has been mainly   studied in the case of one space variable and of fluxes  with just one point discontinuity (but the analysis can be easily extended to the case of finitely many discontinuity points).  Assuming that the discontinuity is located at \(x=0\) and
imposing the validity of  Kruzhkov entropy inequalities
separately on $(-\infty, 0)$ and $(0,+\infty)$,
one can show that  every pair $u$, $v$ of bounded solutions satisfies
\begin{equation}\label{f:katou}
\int 
|u(T,x)-v(T,x)| \, dx \leq 
\int 
|u(0,x)-v(0,x)| \, dx
+ \int_0^T  \W(u^\pm(t,0), v^\pm(t,0)) \, dt
\,,
\end{equation}
where $\W$ is a quantity that depends only on the traces
$u^\pm$, $v^\pm$ of $u$ and $v$ at $x=0$. 
The \(L^1\)-contractivity of the semigroup associated with~\eqref{f:evo0} is then obtained if $W(u^\pm, v^\pm)\leq 0$ for every pair of solutions.
%
Several conditions  have been proposed in literature in order to have that \(W\le0\), and different conditions lead to different physically relevant semigroups of solutions, see \cite{AKR,GNPT}.

In \cite{AKR} Andreianov, Karlsen and Risebro  have proposed a general framework in order to study uniqueness for  \eqref{f:evo0} in the model  case of one space variable and for fluxes with finitely many discontinuity points. 
The validity of the inequality \(W\le 0\) is axiomatized in the notion of \(L^1\)-dissipative germ and given a germ \(\mathcal G\)  they show uniqueness of \(\mathcal G\)-entropy solutions, see Definition~3.8 
in~\cite{AKR} and Definition~\ref{def:gentropy} below.

Loosely speaking, at a point of discontinuity of  the flux \(\F\), a germ \(\mathcal{G}\) is a set of pairs \((u^-,u^+)\)
satisfying the Rankine--Hugoniot condition, 
such that
\[
W(u^\pm, v^\pm) \leq 0 \qquad
\forall (u^-, u^+),\ (v^-,v^+) \in \mathcal{G}
\]
and, in the model case of flux with one single discontinuity at \(x=0\), a \(\mathcal G\)-entropy solution is a solution of \eqref{f:evo0} satisfying  Kruzkov's conditions outside the origin and whose traces at \(0\) belong to \(\mathcal G\).  A similar analysis has been performed, always in the model case of one dimensional fluxes with one discontinuity point, independently by Garavello, Natalini, Piccoli and Terracina in \cite{GNPT} in terms of the notion of dissipative Riemannian solvers. 
Let us  also mention that this analysis can be extended to the multidimensional case by assuming that the set of discontinuity of the flux is a regular submanifold, see \cite{AnMitr}, or by assuming a priori \(BV\) regularity of the solution, see \cite{CDD}. 

\smallskip
The main purpose of this paper is to provide a general framework  to extend this analysis to solutions of \eqref{f:evo0} under quite general assumptions on the  flux. In order to do this we introduce a rather weak notion  of entropy solution, see Definition~\ref{d:entrsol} below, and under a suitable genuine nonlinearity assumption on the flux we show that these solutions admits traces on the discontinuity set of the coefficients, see Theorem \ref{thm:tr0} below. 
Once the existence of traces  has been established we prove that any pair of weak entropy solutions of  \eqref{f:equa} satisfies a generalized Kato inequality with a reminder term concentrated on the discontinuity set of the flux, see Theorem~\ref{thm:Kato0} below. 
It is then classical to show that this Kato type inequality leads to a quasi contractivity inequality for solutions of \eqref{f:evo0} of the form \eqref{f:katou}. 
Once this inequality has been established, the analysis in \cite{AKR} in terms of germs and of 
\(\mathcal G\)-entropy solutions can be straightforwardly extended to~\eqref{f:evo0}, 
see Theorem~\ref{thm:gentropy} below. 
As a byproduct of our results we can also obtain existence and uniqueness of solutions of~\eqref{f:evo0} assuming Sobolev dependence of the flux \(\F\) with respect to \((t,x)\), see Theorem~\ref{thmW110} below.

\medskip
Let us now describe in a more detailed way our main results.
First of all, 
the structural assumptions on $\A$
and the results in \cite{ACDD} guarantee
the existence of a  \(\H^{n-1}\)--rectifiable set \(\mathcal N\)
(defined in~\eqref{def:N} below)
that represents a universal jump set of \(\A(\cdot, v)\), independent of \(v\).

We say that a distributional solution
\(u\in L^{\infty}(\R^n)\) of~\eqref{f:equa}
a \textsl{weak entropy solution} (WES) of~\eqref{f:equa},
if there exists
a non-negative Radon measure \(\mu\) such that
\(\mu(\R^n\setminus \mathcal N)=0\) and, for every \(k\in\R\),
\begin{equation}\label{eq:ent0}
\Div_z \Big(\sign (u-k) [\A(z,u)-\A(z,k)]\Big)+\sign(u-k)\Div^a_z\A(z,k)\leq \mu,
\end{equation}
see Definition~\ref{d:entrsol} below.
Here \(\Div^a_z\A(\cdot,k)\) denotes, for every \(k\in\R\),
the absolutely continuous part of the measure
\(\Div_z \A(\cdot,k)\).

As we shall see in a moment,
the notion of weak entropy solution is strong enough
to guarantee that such solutions possess a reasonable structure.
On the other hand, it is weak enough to include
essentially all solutions of~\eqref{f:equa}
obtained by approximation schemes.
In particular, under our assumptions on the flux, the solutions constructed by Panov
in \cite{Pan1} are weak entropy solutions.

Assuming the genuine nonlinearity of the flux,
and adapting to our setting
the techniques developed by
De Lellis, Otto and Westdickenberg in \cite{DLOW}
(see also \cite{Pan2007,Vas}),
our first result ensures
the existence of traces on $\mathcal{N}$
for weak entropy solutions.  Loosely speaking, we have the following result, 
see Theorem~\ref{thm:tracce} below for the precise statement.

\begin{theorem}[Existence of traces]\label{thm:tr0}
If \(u\) is 
a bounded weak entropy solution of~\eqref{f:equa},
then \(u\) admits traces \(u^{\pm}\) 
on \(\mathcal{N}\)
(in a generalized sense, 
see Definition~\ref{def:tracce}).
\end{theorem}

%

\medskip
The existence of generalized traces 
of weak entropy solutions
allows us to prove the validity
of the following
\textsl{Generalized Kato Inequality},
see Theorem~\ref{thm:representation} below for the precise statement.

\begin{theorem}[Generalized Kato Inequality]\label{thm:Kato0}
	Let \(u\) and \(v\) be weak entropy solutions.
	Then  
	\begin{equation}\label{Kato0}
	\Div_z \Big(\sign (u-v) [\A(z,u)-\A(z,v)]\Big)\leq W(u^\pm, v^\pm)\, \mathcal H^{n-1}\res \mathcal N\,,
	\end{equation}
where
\begin{equation}\label{W0}
\begin{split}
W(u^\pm,v^\pm)
&=\Big\{\sign (u^+-v^+) [\A^+(z,u^+)-\A^+(z,v^+)]\\
&\quad\quad-(\sign (u^--v^-) [\A^-(z,u^-)-\A^-(z,v^-)]\Big\}\cdot \nu_{\mathcal{N}},
\end{split}
\end{equation}
where $\nu_{\mathcal{N}}$ is the measure--theoretic normal to the
$\mathcal{H}^{n-1}$ rectifiable set $\mathcal{N}$,
and $\A^\pm(z,v)$ are the traces at $z$ of the
SBV function $z\mapsto \A(z, v)$.
\end{theorem}

In order to prove the above theorem,
we combine Kruzkov's doubling of variables technique (see \cite{Kruz}) 
with Ambrosio's lemma on incremental quotients of $BV$ functions
(see \cite{Ambrosio2004}),
to show that the left--hand side of \eqref{Kato0} is a measure
whose positive part is concentrated on $\mathcal N$.
Once this result has been established, 
the representation formula~\eqref{W0} is an easy consequence
of the existence of traces.

\medskip
Our main application concerns the study of uniqueness conditions for the 
Cauchy problem associated to the
multidimensional evolutionary equation~\eqref{f:evo0}. 
In this case, a bounded distributional solution
\(u\in L^{\infty}((0,+\infty)\times\R^N)\) of~\eqref{f:evo0}
is a weak entropy solution to~\eqref{f:evo0} if,  for every \(k\in\R\),
\begin{equation}\label{eq:ent1}
\partial_t |u-k| + \Div_x \Big(\sign (u-k) [\F(t,x,u)-\F(t,x,k)]\Big)+\sign(u-k)\Div^a_x\F(t,x,k)\leq \mu,
\end{equation}
where \(\mu\) is, as before, a non-negative measure concentrated on \(\mathcal{N}\).

The Generalized Kato Inequality \eqref{Kato0}
implies the \textsl{quasi}--contractivity of the \(L^1\) norm of 
the difference of solutions
in the following sense: 
if 
\(u, v \in C^{0}([0,+\infty); L^1(\R^N))\cap L^\infty((0,+\infty)\times\R^N)\) 
are weak entropy solutions of~\eqref{f:evo0},
then for every \(T>0\) and every \(R>0\)
\begin{equation}\label{f:contr}
\begin{split}
& \int_{B_R}\ 
|u(T,x)-v(T,x)| \, dx
\\ & \leq 
\int_{B_{R+VT}}\ 
|u(0,x)-v(0,x)| \, dx
+ \int_{{\mathcal N}\cap ([0,T]\times B_{R+VT})} \W(u^\pm,v^\pm)\,d\H^{N} 
\,,
\end{split}
\end{equation}
where \(B_r := \{x\in\R^N:\ |x|< r\}\) and \(V := \|\A\|_{\infty}\).
As a consequence, if one prescribes an entropy condition stronger than~\eqref{eq:ent0} and
implying the inequality \(W\leq 0\), then
the Generalized Kato Inequality would give the 
standard contractivity inequality
\begin{equation}\label{f:contr2}
\int_{B_R}\ 
|u(T,x)-v(T,x)| \, dx
\leq 
\int_{B_{R+VT}}\ 
|u(0,x)-v(0,x)| \, dx
\end{equation}
and hence the uniqueness of
solutions to the Cauchy problems associated to~\eqref{f:evo0}, see Definition~\ref{def:germ} 
and Theorem~\ref{thm:gentropy} below. 

\medskip

Let us also stress that existence of solutions 
satisfying these additional entropy conditions is not trivial and currently 
not known in the general setting here considered. 
Existence results are available assuming additional conditions on the structure of the flux field, see Remark~\ref{rmk:onesto} for a more detailed discussion.
\medskip

In case \(\F(\cdot,u)\in W^{1,1}\) and satisfies the assumptions listed in Section~\ref{sec:prel}, it is straightforward to check that \(\mathcal N=\emptyset\), 
so that~\eqref{Kato0} implies \emph{contractivity} of the semigroup associated to~\eqref{f:evo0}. 
In particular, also using the results of Panov \cite{Pan1}, we can  generalize to this situation  the classical Kruzkov results concerning existence and uniqueness  of solutions of \eqref{f:evo0}, see Theorem~\ref{thmW110} and Remark~\ref{rmk:onesto} below.

\medskip

Let us conclude this Introduction by presenting the structure of the paper. 
In Section~\ref{sec:prel} below we state our main assumption on the flux \(\A\), we recall some of its consequence and we provide the precise statements of our main results. 
In Section~\ref{s:tracce} we prove Theorem~\ref{thm:tracce}, 
in Section~\ref{sec:representation} we prove Theorem ~\ref{thm:representation} and eventually 
in Section~\ref{sec:fine} we provide the proofs of Theorems~\ref{thm:gentropy} and~\ref{thmW110}.

\medskip

\begin{ack}
F.G.\ has been supported by ERC 306247 \textit{Regularity of area-minimizing currents} and by
SNF 146349 {\it Calculus of variations and fluid dynamics}. 
G.D.P.\ is supported by the MIUR SIR grant \textit{Geometric Variational Problems} (RBSI14RVEZ).
\end{ack}

\medskip
\noindent
\textit{Conflict of interest:}
The authors declare that they have no conflict of interest.

\section{Assumptions on the vector field and main results}\label{sec:prel}
In this section we   state our main structural hypotheses  on the vector field (assumptions (H1)--(H5) below)  and prove some consequences of these assumptions.

\subsection{Structural assumptions on the vector field}
  Let \(\A\in L^\infty (\R^n\times \R;\R^n)\) be such that:
\begin{enumerate}
\item[(H1)] There exists a set \(\mathcal C_{\A}\) with \(\Leb{n}(\mathcal C_{\A})=0\) such that 
 \(\A(z,\cdot)\in C^1(\R,\R^n)\) for every \(z\in \R^n\setminus  \mathcal C_{\A}\) and 
 \(\A(\cdot,v)\in SBV(\R^n,\R^n)\) for every \(v\in \R^n\).
\item[(H2)]There exists a constant \(M\) such that
 \[
 |\partial_v \A(z, v)|\le M\qquad  \forall \, z\in \R^n\setminus \mathcal C_{\A},\quad v\in \R.
 \]

\item[(H3)]There exists a modulus of continuity \(\omega\) such that 
\[
|\partial_v\A (z, u)-\partial _v \A (z,w)|\le \omega(|u-w|)\qquad \forall\, z\in \R^n\setminus\mathcal  C_{\A},\quad u\,,w \in \R.
\]
\item[(H4)]There exists a function \(g\in L^1(\R^n)\)  such that  
\[
\begin{split}
|\nabla_z \A(z,u)-\nabla_z \A(z,w)|&\le g(z)|u-w|
\qquad
 \forall\, z\in \R^n\setminus\mathcal  C_{\A},\quad u\,,w \in \R,
\end{split}
\] 
where \(\nabla_z \A(z,v)\)
denotes the approximate gradient of the map \(z\mapsto \A(z,v)\).
 \item[(H5)] The measure
\begin{equation}\label{sigma}
\sigma := \bigvee_{u\in \R} |D_z \A(\cdot, u)|
\end{equation}
satisfies \(\sigma(\R^n)<\infty\). 
Here \(D_z \A(\cdot, u)\) is the distributional gradient of the map  \(z\mapsto \A(z,u)\) (which is a measure since \(\A(\cdot ,u)\in BV\)) and \(\bigvee\) denotes the least upper bound in the space of non-negative Borel measures, see \cite[Definition~1.68]{AFP}.
\end{enumerate}


 Assumptions (H1)-(H5) imply that  \(\A\) satisfies the hypotheses of \cite{ACDD}. Let us summarize some consequence of this fact. First of all from the definition of $\sigma$ we deduce that 
\[
\bigvee_{v\in\R}|\nabla\A(\cdot,v)|\L^n\leq \sigma^a\L^n\qquad \bigvee_{v\in\R}|D^s\A(\cdot,v)|\leq \sigma^s,
\] 
where $\nabla\A(\cdot,v)$ and $D^s\A(\cdot,v)$ are the approximate differential of $\A(\cdot,v)$ and the singular part of the measure
$D\A(\cdot,v)$ respectively, and $\sigma^a\L^n$ and $\sigma^s$ are the the absolutely continuous and singular parts of $\sigma$.
Moreover if we define 
\begin{equation}\label{def:N}
\mathcal N := \Big\{z\in \R^n: \liminf_{r \to 0} \frac{\sigma(B_r(z))}{r^{n-1}}>0\Big\},
\end{equation}
then \(\mathcal N\) is a  \(\H^{n-1}\)  rectifiable set\footnote{Recall that a set $\mathcal N\subset\R^n$ is said $\H^{n-1}$-rectifiable (shortened: rectifiable) if there are countably many $C^1$ submanifolds $M_i$ of dimension $n-1$ such that $\mathcal H^{n-1}(\mathcal N \setminus \bigcup_i M_i)=0$.},
see Section~3 in \cite{ACDD}. 
Furthermore for \(\H^{n-1}\)-a.e. point in \(\R^n\setminus \mathcal N\) and {\em every} \(v\in \R\) there exists  the limit
\[
\tilde{\A}(z,v) := \lim_{r\to 0}\mean{B_r(z)} \A(y,v)dy\,,
\]
and  for \(\H^{n-1}\) almost every \(z\in \mathcal N\) and \emph{every} \(v\in \R\) there exists the traces of \(\A\) on \(\mathcal N\) defined as:  
\begin{equation}\label{eq:defApm}
\A^{\pm}(z,v) := \lim_{r\to 0} \mean{B^\pm _r(z)} \A(y,v) dy,
\end{equation}
where we denoted $B^\pm _r(z) = \{w\in B_r(z) : \pm\langle w-z,\nu(z)\rangle\geq 0\}$. 
In addition  the functions \(v\mapsto \tilde{\A}(z,v),\,\A^{\pm}(z,v)\) are \(C^1\) with derivatives given by \(\partial_v \widetilde \A(z,v)= \widetilde {\partial_v \A}(z,v)\) and 
\(\partial_v \A^{\pm}(z,v)=(\partial_v \A(z,v))^{\pm}\) respectively,
 see \cite[Proposition 3.2]{ACDD}. 
 Hence, if we denote by \(\a\) the vector field
\begin{equation}\label{defapiccolo}
\a(z,v) := \partial_v \A(z,v),
\end{equation}
then \(\a\) admits a precise representative for \(\H^{n-1}\)-almost every \(z\in \R^n\setminus \mathcal N\) as well as one sided traces on \(\mathcal N\) that agree with \(\partial_v\A\) (respectively with \(\partial_v \A^{\pm}\)).

In the sequel we shall assume the following genuine nonlinearity hypothesis: 
\begin{equation}\tag{GNL}
\mathcal L^{1}(\{v: \a^{\pm}(z,v) \cdot \xi=0\})=0 \qquad \textrm{for every \(\xi\in S^{n-1}\) and for \(\H^{n-1}\) a.e. \(z\in \mathcal N\)}.
\end{equation}

\begin{remark}\label{lincei}
	Let us  point out that our hypotheses include (and actually are modeled on) the case \(\A(z,v)=\widehat \A (w(z),v)\) where \(w\in SBV(\R^n; \R^d)\cap L^\infty(\R^n; \R^d) \),  
	\(\widehat \A\in C^1(\R^d\times \R,\R^n)\cap{\rm Lip} (\R^d\times \R,\R^n) \),
	and 
\[
\mathcal L^{1}(\{v: \partial_v{\widehat \A}(w,v) \cdot \xi=0\})=0 \qquad 
\textrm{for every \(\xi\in S^{n-1}\) and for every \(w\in \R^d\)}.
\]	
\end{remark}

\begin{remark}
Since we are dealing with bounded solutions,
all our assumptions can be localized in the $v$ variable.
Moreover, it is not difficult to modify the proofs in order to
localize also in the $z$ variable, 
see Remark 3.5 in \cite{ACDD}.
\end{remark}

\subsection{Main results}
We consider the following scalar conservation law
\begin{equation}\label{f:scalar}
\Div_z \A(z,u(z))=0\,,
\end{equation}
where
$\A\colon\R^n\times\R\to\R^n$ satisfies the structural assumption (H1)--(H5) and (GNL).

\begin{definition}[Weak entropy solutions]\label{d:entrsol}
A function  \(u\in L^{\infty}( \R^n)\)
is a {\em  weak entropy solution} (WES shortened) of \eqref{f:scalar}
if $u$ is a distributional solution of \eqref{f:scalar} and for every \(k\in \R\) it holds
\begin{equation}\label{eq:ent}
\Div_z \Big(\sign (u-k) [\A(z,u)-\A(z,k)]\Big)+\sign(u-k)\Div^a_z\A(z,k)\leq \mu,
\end{equation}
where \(\mu\) is a non-negative Radon measure {\em independent} of \(k\) and  such that \(\mu(\R^n\setminus \mathcal N)=0\).
\end{definition}
Here 
$\diver^a_z\A(z,k) = {\rm tr}\nabla \A(z,k)$ is the absolutely continuous part of $\diver_z\A(\cdot,k)$.


\begin{definition}[Traces]\label{def:tracce}
Let $u\in L^\infty(\R^n)$ and  
let $\mathcal J\subset\R^n$ be an $\H^{n-1}$-rectifiable set oriented by a normal vector
field $\nu$. We let the set of traces of $u$ at $z_0\in\mathcal J$ be
\[
\Gamma_{u, \mathcal J}(z_0) := \Big\{ (c^-,c^+): \exists r_k\downarrow 0 : u_{z_0,r_k}\rightarrow c^-\one_{H^-} + c^+\one_{H^+} \text{ in }L^1_{\rm loc}   \Big\},
\]
where $ u_{z_0,r_k} (z) :=  u(z_0 + r_k(z-z_0))$, 
$H^\pm :=\{ z\in\R^n : \pm \langle z-z_0,\nu\rangle \geq 0 \}$
and $\one_A$ denotes the characteristic function of a set $A$. 
\end{definition}


The very same definition can be given component-wise for a vector field \(\boldsymbol B\in L^\infty(\R^n, \R^n)\). 
Moreover it is immediate to see from the definition that if $u\in L^\infty(\R^n)$ and \(f\in C^0(\R)\) then
\[
\Gamma_{f(u), \mathcal J}(z_0) = \big\{(f(c^-),f(c^+)) : (c^-,c^+)\in  \Gamma_{u, \mathcal J}(z_0)  \big\} .
\]

\begin{theorem}[Existence of generalized traces]\label{thm:tracce}
If \(u\) is a WES, then for $\mathcal H^{n-1}$ almost every $z_0\in\mathcal N$ 
\[
\Gamma_{u,\mathcal N}(z_0)\neq\emptyset.
\]
Moreover if  $(c^-,c^+)\in \Gamma_{u,\mathcal N}(z_0)$ satisfies $c^-\neq c^+$ then the traces are unique:
$\Gamma_{u,\mathcal N}(z_0) = \{    (c^-,c^+)\}$.  Otherwise there exist $a,b\in\R$ such that 
\[
\Gamma_{u,\mathcal N}(z_0) = \{ (v,v): v\in [a,b] \}.
\]
Finally, the Rankine--Hugoniot condition holds: 
\[
\A^-(z_0,c^-)\cdot \nu(z_0)=\A^+(z_0,c^+)\cdot \nu(z_0)\qquad \forall\, (c^-,c^+)\in \Gamma_{u,\mathcal N}(z_0).
\]  

\end{theorem}

\bigskip
\begin{theorem}[Generalized Kato Inequality]
\label{thm:representation}
Let \(u\) and \(v\) be WES.
Then  
there exists a Borel function $w\colon\mathcal{N}\to\R$ such that
the following Kato inequality holds true:
\begin{equation}\label{Kato2}
\Div_z \Big(\sign (u-v) [\A(z,u)-\A(z,v)]\Big)\leq w\, \mathcal H^{n-1}\res \mathcal N\,.
\end{equation}
Furthermore, for \(\mathcal H^{n-1}\) almost every \(z\in \mathcal{N}_1 := \{z\in\mathcal{N}:\ w(z)\neq 0\}\),
the functions $u$ and $v$ admit
unique traces at $z$ and the following representation formula holds:
\begin{equation}\label{W}
\begin{split}
w = W(u^\pm,v^\pm)
&=\Big\{\sign (u^+-v^+) [\A^+(z,u^+)-\A^+(z,v^+)]\\
&\quad\quad-(\sign (u^--v^-) [\A^-(z,u^-)-\A^-(z,v^-)]\Big\}\cdot \nu.
\end{split}
\end{equation}
\end{theorem}

The Generalized Kato Inequality yields a uniqueness result for the Cauchy problem
for the evolutionary equation
\begin{equation}\label{f:evo}
\begin{cases}
u_t + \Div_x \A(t,x,u) = 0,& \text{in}\ (0,+\infty)\times \R^N,\\
u(0,x) = u_0(x), & x\in\R^N\,.
\end{cases}
\end{equation}
More precisely, if we prescribe an entropy condition stronger than~\eqref{eq:ent} and
implying the inequality \(w\leq 0\), then
the Generalized Kato Inequality gives the uniqueness of
solutions to~\eqref{f:evo}.  
To this end let us recall the definition of \textsl{dissipative germ} introduced in \cite{AKR}, 
see~Definition 3.1 there.

\begin{definition}[Germ]\label{def:germ}
Given two functions \(f^\pm\in C^0(\R)\), a set \(\mathcal G\subset \R^2\) is said to be a dissipative germ associated to \(f^\pm\) if the following two conditions hold true:
\begin{itemize}
\item[(i)] Every \((u^-,u^+)\in \mathcal G\) satisfies the Rankine--Hugoniot condition \(f^+(u^+)=f^-(u^-)\).
\item[(ii)] For every two pairs \((u^-,u^+)\,,(v^-,v^+)\in \mathcal G\) we have
\[
W_{f^\pm}(u^\pm,v^\pm):=\Big\{\sign (u^+-v^+) [f^+(u^+)-f^+(v^+)]-(\sign (u^--v^-) [f^-(u^-)-f^-(v^-)]\Big\}\le 0.
\]
\end{itemize}
\end{definition}

Following \cite{AKR1} we now define \(\boldsymbol{\mathcal G}\)-entropy solutions associated to germs,  
 compare with Definition~3 there and Definiton~3.8 in \cite{AKR}. 

\begin{definition}[$\boldsymbol{\mathcal G}$-entropy solutions]\label{def:gentropy} 
	Let \(\F: \R\times \R^N \times \R \to \R^N\) be such that \(\A:=(u,\F)\) satisfies (H1)-(H5) and (GNL) above. Let \(\mathcal N\subset \R\times \R^N\) be the rectifiable set defined in \eqref{def:N}. Assume that  for every \(z=(t,x)\in \mathcal N\)  such that \(\A^{\pm}(z,u)\) exist it is given a dissipative germ \(\mathcal G_{z}\) associated to \(f^{\pm}(u):=\A^{\pm}(z,u)\cdot\nu(z)\) and let us set \(\boldsymbol{\mathcal G}=\{\mathcal G_z\}_{z\in \mathcal N}\). 
	We say that a bounded function \(u\in C^0([0,+\infty);L^1(\R^N))\) is a \(\boldsymbol{\mathcal G}\)-entropy solution of  \eqref{f:evo} if 
\begin{itemize}
\item[(i)] \(u\) is a weak entropy solution of \eqref{f:evo} according to Definition \ref{d:entrsol}.
\item[(ii)] For \(\mathcal H^{N}\)-almost every \(x\in \mathcal N\) any \((u^-,u^+)\in \Gamma_{u,\mathcal N}(z) \) belongs to the germ \(\mathcal G_z\).
\end{itemize}
\end{definition}

A straightforward consequence of  Theorem \ref{thm:representation} is then the following:

\begin{theorem}[Uniqueness of \(\boldsymbol{\mathcal G}\)-entropy solutions]\label{thm:gentropy}
Let \(\F: \R\times \R^N\times\R\to \R^N\) be such that \(\A:=(u,\F)\) satisfies (H1)-(H5) and (GNL) above. Then for any choice of \(\boldsymbol{\mathcal G}\) there exists \emph{at most} one \(\boldsymbol{\mathcal G}\)-entropy solution of \eqref{f:evo}.
\end{theorem}

\begin{remark}\label{rmk:onesto}
Under  mild requirements on the flux, 
the existence of weak entropy solutions can be obtained by the results of Panov, see \cite{Pan1}. 
On the other hand, the existence of \(\boldsymbol{\mathcal G}\)-entropy solutions, i.e.\ additionally satisfying 
condition~(ii) in Definition~\ref{def:gentropy}, is far from trivial and known only in some special cases. 
Positive results in this direction are available either in one space dimension for a flux with a finite number of discontinuity points, see  for instance \cite{AKR,GNPT} and the references therein, or in many space dimensions and for the particular case of the 
 \textsl{vanishing viscosity germ}, assuming that the jump set of the \( \F \) is a $C^2$ submanifold \cite{AKR1}, see also  \cite{AnMitr} where a more general situation is considered. 
%
\end{remark}

If \(\F(\cdot,u)\) is a Sobolev function one can easily obtain from the above analysis uniqueness of (weak) entropy solutions.
\begin{theorem}\label{thmW110}
Let \(\F: \R\times \R^N\times\R\to \R^N\) be such that \(\A:=(u,\F)\) satisfies (H1)-(H5) and (GNL) above and assume that \(\F(\cdot,u)\in W^{1,1}(\R\times \R^N, \R^N)\) for every \(u\in \R\).  
Then any 
two (weak) entropy solutions 
\(u, v \in C^{0}([0,+\infty); L^1(\R^N))\cap L^\infty((0,+\infty)\times\R^N)\) 
of~\eqref{f:evo0} satisfy
\begin{equation*}
 \int_{\R^N}\ |u(T,x)-v(T,x)| \, dx \leq \int_{\R^N}\ |u(0,x)-v(0,x)| \, dx\,.
\end{equation*}
\end{theorem}

\section{Proof of Theorem~\ref{thm:tracce}}\label{s:tracce}
In this section we prove Theorem~\ref{thm:tracce}. We start with the following well known Lemma.
\begin{lemma}\label{lemma:campo}
Let $\boldsymbol B \in L^\infty(\R^n,\R^n)$ and assume that $\mu=\Div \boldsymbol B$ is a Radon measure. 
Then $|\mu|\ll\mathcal H^{n-1}$. Furthermore if $\mathcal J$ is a rectifiable set and $\Gamma_{\boldsymbol B,\mathcal J}(z)\neq\emptyset$ for \(\mathcal H^{n-1}\) almost every $z\in\mathcal J$ then
it holds
\[
\Div (\boldsymbol B)\res \mathcal J =  (\B^+ - \B^-) \cdot \nu\, \mathcal H^{n-1}\res \mathcal J
\]  
where \((\B^-(z),\B^+(z))\in \Gamma_{\boldsymbol B,\mathcal J}(z)\). 
In particular for every two pairs in \(\Gamma_{\boldsymbol B,\mathcal J}(z)\), their projections along $\nu(z)$ have the same difference.
\end{lemma}
\begin{proof}
The fact that \(|\mu|\ll \mathcal H^{n-1}\) is proved for instance in \cite[Lemma 2.4]{DL}. To show the second part we decompose   \(\mu\) as 
\[
\mu=\mu\res {\mathcal J}  +\mu\res (\R^n \setminus \mathcal J) =: \mu_1+\mu_2
\]
with  \(\mu_1\perp \mu_2\). Since \(\mu_1\) is a Radon measure and  \(\mathcal H^{n-1}\res J\) is \(\sigma\)-finite we can apply the Radon-Nikodym Theorem to get that
\[
\mu_1= \Div (\boldsymbol B)\res \mathcal J=h \mathcal H^{n-1}\res \mathcal J
\]
for some \(h\in L^1( \mathcal H^{n-1}\res \mathcal J)\). Let now \(z_0\) be  a point such that \(\Gamma_{\boldsymbol B,\mathcal J}(z_0)\ne \emptyset\), 
\[
h(z_0+rz) \mathcal H^{n-1}\res \frac{\mathcal J-z_0}{r}\weakstarto h(z_0) \mathcal H^{n-1}\res \{z\cdot \nu(z_0)=0\}
\]
and 
\[
\lim_{r\to 0}\frac{ |\mu_2|(B_r(z_0))}{r^{n-1}}=0.
\]
Note that \(\mathcal H^{n-1}\) almost every point  \(z\) satisfies the above properties. Indeed, the first one follows by our assumptions, while the second and the third   ones follow, respectively,  from \cite[Theorem 2.83]{AFP}  and \cite[Equation 2.41]{AFP}. 

Let us 
 choose \(r_k\downarrow 0\) with
\[
\B_{r_k}\to \B^-(z_0)\one_{H^-}+\B^+(z_0)\one_{H^+},
\]
where $H^\pm = \{\pm\langle z,\nu(z_0)\rangle\geq 0\}$.
Let $\varphi\in C^1_c(\R^{n})$ and define \(\varphi_{r_k}(z)=r_k^{1-n}\varphi((z - z_0)/r_k)\). Integrating by parts  we get 
\begin{equation}\label{sticazzi2}
\langle \mu, \varphi_{r_k}\rangle=\frac{1}{r_k^n}\int \B(z)\cdot \nabla \varphi \Big(\frac{z - z_0}{r_k}\Big) dz=\int \B_{r_k}(z)\cdot \nabla \varphi(z)\,dz.
\end{equation}
Moreover
\[
\begin{split}
\langle \mu, \varphi_{r_k}\rangle&=
\frac{1}{r_k^{n-1}}\left\langle\mu_1, \varphi \Big(\frac{\cdot - z_0}{r_k}\Big)\right\rangle+\frac{1}{r_k^{n-1}}\left\langle\mu_2, \varphi \Big(\frac{\cdot- z_0}{r_k}\Big)\right\rangle\\
&=\int_{\frac{\mathcal J - z_0}{r_k}}h(r_k z + z_0)\varphi(z)d\mathcal H^{n-1}(z)+O\Big(\frac{ |\mu_2|(B_{r_k}(z_0))}{{r_k}^{n-1}}\Big)\,.
\end{split}
\]
Hence, passing to the limit as \(k\) goes to infinity in \eqref{sticazzi2},
we get 
\[
h(z_0) \int_{\{z \cdot \nu(z_0) = 0\}} \varphi(z) \, d\H^{n-1}(z)
= \B^-(z_0) \cdot \int_{H^-} \nabla\varphi(z)\, dz
+ \B^+(z_0) \cdot \int_{H^+} \nabla\varphi(z)\, dz\,.
\]
Integrating by parts
we obtain that $h(z_0) = (\B^+(z_0) - \B^-(z_0)) \cdot \nu(z_0)$,
and this concludes the proof.
\end{proof}

\begin{proof}[Proof of Theorem~\ref{thm:tracce}]
We divide the proof in several steps.

\medskip\noindent
\textsl{Step 1 (Definition of the measure for the kinetic equation)}.
Let $u$ be a WES, according to \eqref{eq:ent} for every $k\in\R$ the distribution
\begin{equation}\label{eq:ent2}
\eta_k:=\Div_z \Big(\sign (u-k) [\A(z,u)-\A(z,k)]\Big)+\sign(u-k)\Div^a_z\A(z,k)
\end{equation}
is a Radon measure. We now claim that for every $K\Subset\R$ and for every $R>0$
\begin{equation}\label{bbq}
\sup_{k\in K}|\eta_k |(B_R) \leq C(K,R).
\end{equation}
To see this note that  $\mu- \eta_k\geq 0$ for every $k\in K$.
Therefore if $\phi\in C^1_c(B_R)$ and $\chi\in C^1_c(B_{R+1})$ satisfies 
$\chi \geq 0$, $\chi \equiv 1$ in $B_R$, we have
\[
 \langle \mu- \eta_k, (\|\phi\|_\infty \pm \phi)\chi \rangle \geq 0,
\]
hence, since $\chi\phi = \phi$,
\[
\pm \langle \eta_k,\phi\rangle \leq -\langle \eta_k, \chi\rangle \|\phi\|_\infty + 2 \langle\mu,\chi\rangle \|\phi\|_\infty.
\]
The above inequality implies the validity of \eqref{bbq},
since, by the very definition of $\eta_k$,
one has
$\sup_{k\in K} |\langle \eta_k, \chi\rangle | \leq C(K,R)$.
In particular the map 
\[
C^{\infty}_c(\R^n\times\R) \ni \Phi \mapsto \langle\eta,\, \Phi\rangle := \iint_{\R^n\times \R} \Phi(z,k)\, d\eta_k(z)\, dk
\]
defines a Radon measure $\eta$ in $\R^n\times\R$. 
Moreover if we define 
\footnote{Recall that given a Borel measure \(\eta\) on a space 
\(X\) and a Borel map \(\pi:X\to Y\) the measure 
\(\pi_\#\eta\) on Y is defined as \(\pi_\#\eta(U)=\eta(\pi^{-1}(U))\) for every Borel set \(U\subset Y\).} 
$\mesnu:=\pi_\#(|\eta|)$, where $\pi:\R^n\times\R\rightarrow\R^n$ is the projection on the first factor,
then \(\mesnu\ll \mathcal H^{n-1}\). 
Indeed by Lemma \ref{lemma:campo} $|\eta_k|\ll\mathcal{H}^{n-1}$ so that if $\mathcal{H}^{n-1}(A)=0$ then 
\[
 \mesnu(A) \leq \int_{\R} |\eta_k|(A) dk =0. 
\]

\bigskip\noindent
\textsl{Step 2 (Kinetic formulation)}.
The function $(k,z)\mapsto \chi(k, u(z)) := \sign(u(z) - k)$
is a solution of the kinetic equation, see \cite{LPT} 
\begin{equation}\label{f:kinetic}
\diver_z \left[\chi(k, u) \partial_v \A(z,k)\right]
- \partial_k\left[\chi(k,u)\diver_z^a \A(z,k)\right]
= -\partial_k \eta
\qquad\text{in}\ \mathcal{D}'(\R^{n+1}),
\end{equation}
where \(\eta(k, A) := \eta_k(A)\).  
Indeed, let us consider in equation~\eqref{eq:ent2} a test function of the form
$\Phi(k,z) := \varphi(z) \partial_k \psi(k)$.
Recalling the definition of the measure $\eta$ and of $\chi(k,u)$, integrating in $k$ we get
\[
\begin{split}
& -\iint \nabla\varphi(z) \partial_k\psi(k) \chi(k,u) [\A(z,u) - \A(z,k)]\, dz\,dk
\\ &
+\iint \varphi(z) \partial_k\psi(k) \chi(k,u) \diver_z^a\A(z,k)\, dz\, dk
 = \int \varphi(z) \partial_k\psi(k)\, d\eta(k,z),
\end{split}
\]
so that
\[
\begin{split}
& \iint \nabla\varphi(z) \psi(k) \partial_k\big(\chi(k,u) [\A(z,u) - \A(z,k)]\big)\, dz\,dk
\\ &
-\iint \varphi(z) \psi(k) \partial_k\big(\chi(k,u) \diver_z^a\A(z,k)\big)\, dz\, dk
= -\int \varphi(z) \psi(k)\, d\partial_k\eta(k,z).
\end{split}
\]
Since  the function \(k\mapsto \chi(k,u) [\A(z,u) - \A(z,k)]\) is Lipschitz, it is straightforward to check that $\partial_k\big(\chi(k,u) [\A(z,u) - \A(z,k)]\big) = -\chi(k,u)\partial_v\A(z,k)$,
hence  \eqref{f:kinetic} holds.

\bigskip
\noindent
\textsl{Step 3 (Blow--up)}.
Let 
$\eta(k,z) = \mesnu(z)\otimes \lambda_z(k)$ be the 
disintegration of the measure $\eta$ with respect to $\mesnu$, see \cite[Sect.~2.5]{AFP}.
Since  \(\H^{n-1}\res \mathcal N\) is \(\sigma\)-finite by the Radon-Nikodym Theorem we can write 
\begin{equation}\label{eq:nu}
 \mesnu = h \, \H^{n-1}\res \mathcal N+ \mesnu\res(\R^n\setminus\mathcal N)
\end{equation}
with \(h\in L^1(\H^{n-1}\res \mathcal N)\).
Let us now fix a point $z_0\in\mathcal{N}$ and for $r>0$ let us consider the following rescalings 
in the variable $z$:
\begin{equation}
\begin{gathered}
u_r(z) := u(z_0+rz),
\qquad
\A_r(z,v) := \A(z_0+rz, v),
\\		\label{rescal}
\eta_{k,r}(V) := \frac{\eta_k(z_0+r V)}{r^{n-1}}\,,
\quad
\eta_{r}(U\times V) := \frac{\eta(U\times (z_0+r V))}{r^{n-1}}\,,
\quad U\subset\R,\ V\subset\R^n\ \text{Borel}.
\end{gathered}
\end{equation}
Recall the proof of Lemma \ref{lemma:campo}: for \(\mathcal H^{n-1}\) almost every \(z_0\) in \(\mathcal N\) we have
\begin{equation}\label{eth}
h(z_0+rz) \mathcal H^{n-1}\res \frac{J-z_0}{r}\weakstarto h(z_0) \mathcal H^{n-1}\res \{\nu(z_0)\cdot z=0\}.
\end{equation}
We now claim that for \(\mathcal H^{n-1}\) almost every such \(z_0\) 
and for every $k\in\R$
\begin{equation}\label{conv1}
\begin{gathered}
\A_r(z,k) \to \Abar_{z_0}(z,k) := \A^+(z_0,k) \one_{H^+}(z) + \A^-(z_0,k) \one_{H^-}(z),\\
\partial_v \A_r(z,k) \to \partial_v \Abar_{z_0}(z,k) := \partial_v \A^+(z_0,k) \one_{H^+}(z) + \partial_v \A^-(z_0,k) \one_{H^-}(z),\\
\diver_z^a \A_r(z,k) \to 0\,,
\end{gathered}
\end{equation}
locally in \(L^1(\R^n)\), with   \(H^\pm=\{z: \pm z\cdot \nu(z_0)>0\}\). 
Indeed the first two equations follow directly from the hypotheses on $\A$, see \cite[Proposition 3.2]{ACDD}, while
the last limit in \eqref{conv1}  is a consequence of the fact that $\sup_k|\diver^a\A(z,k)|\leq \sigma^a(z)$ and that  
 \[
 \lim_{r\to 0} \frac{1}{r^{n-1}}\int_{B_r(z_0)}|\sigma^a(z)|\,dz=0\,,
 \]
for \(\mathcal H^{n-1}\) almost every point in \(\mathcal N\), see \cite[Equation 2.41]{AFP}. 
We now prove that, up to $\mathcal H^{n-1}$-negligible subset of $z_0\in\mathcal N$ it holds:
\begin{equation}\label{conv2}
\eta_r\weakstarto  h(z_0)\, \lambda_{z_0}(k)\otimes \H^{n-1}\res \mathcal \partial H^+. 
\end{equation}
To this end observe that by \cite[Equation 2.41]{AFP}, for \(\mathcal H^{n-1}\) almost every \(z_0\in \mathcal N\),
\[
\lim_{r\to 0}\frac{|\mesnu\res (\R^n\setminus \mathcal N)|(B_r(z_0))}{r^{n-1}}=0.
\]
Now it is easy to see that, up to negligible sets,
\begin{equation}\label{tm}
 \Big\{z\in \mathcal N: 0<\limsup_{r\to 0}\frac{|\mesnu\res \mathcal N|(B_r(z))}{r^{n-1}}<\infty\Big\}=\{z \in\mathcal N: h(z)>0\}.
\end{equation}
Since \(\mathcal H^{n-1}\res(\mathcal N\cap  \{h>0\})\ll \mesnu\res  \mathcal N\), \(\mathcal H^{n-1}\) almost every \(z_0\in \mathcal N\cap  \{h>0\}\) is a Lebesgue point for the measure valued map $z \mapsto \lambda_z$ with respect to \(\nu\). 
By combining this with \eqref{eth} one can argue as in Lemma \ref{lemma:campo} to deduce \eqref{conv2} on $\{h>0\}$, see for instance \cite[Proposition 9]{DLOW}. 
Finally by \eqref{tm} we have that \(\mathcal H^{n-1}\) almost every \(z_0\in \mathcal N\) satisfies  \eqref{conv2}, 
since this convergence trivially holds for $\H^{n-1}$ almost every \(z_0\in \{h=0\}\cap \mathcal N\).

\bigskip
\noindent
\textsl{Step 4 (Limiting equation and existence of traces)}. Let us take a point \(z_0\) such that \eqref{conv1} and \eqref{conv2} hold true. According to  Lemma  \ref{l:panov} below,  the sequence \((u_r)_r\) is relatively compact in \(L^1_{\rm loc} (\R^n)\). Let us  now compute the equation satisfied by any cluster point \(u_\infty\) of \((u_r)_r\). To this end, note that  $u_r$ solves
\[
\diver_z \big(\sign(u_r-k) \partial_v \left[\A_r(z, u_r) - \A_r(z,k)\right]\big)
+ \sign(u_r-k)\diver_z^a \A_r(z,k)
= \eta_{k,r}.
\]
Let $(r_j)$ be a sequence converging to $0$ such that 
\(u_{r_j} \to \uinf\) in  \(L^1(B_1)\).
Passing  to the limit 
in the kinetic equation satisfied by the function $(k,z) \mapsto \chi(k, u_{r_j}(z))$, 
\begin{equation*}
\label{f:rkinetic}
\diver_z \left[\chi(k, u_{r_j}) \partial_v \A_{r_j}(z,k)\right]
- \partial_k\left[\chi(k,u_{r_j})\diver_z^a \A_{r_j}(z,k)\right]
= -\partial_k \eta_{r_j}
\quad\text{in}\ \mathcal{D}'(\R^{n+1})\,,
\end{equation*}
and taking into account~\eqref{conv1} and~\eqref{conv2}, we obtain 
\begin{equation}
\label{f:bkinetic}
\diver_z \left[\chi(k, \uinf) \partial_v \Abar_{z_0}(z,k)\right]
= -\partial_k 
\big(h(z_0)\, \lambda_{z_0}(k)\, \H^{n-1}\res \mathcal \partial H^\pm\big)
\qquad\text{in}\ \mathcal{D}'(\R^{n+1}).
\end{equation}
In particular, 
due to the special form \eqref{conv1} of $\Abar_{z_0}$, 
in the half-space $H^{+}$ (resp.\ $H^-$),
equation~\eqref{f:bkinetic} is a transport
equation of the form
\begin{equation}
\label{f:transport}
\a^{+}(k) \cdot \nabla_z \chi(k, \uinf) = 0
\qquad
(\text{resp.}\ \a^{-}(k) \cdot \nabla_z \chi(k, \uinf) = 0),
\end{equation}
where
\[
\a^{\pm}(k) := \partial_v\A^{\pm}(z_0,k).
\]
Since, by (GLN), these vector fields are genuinely nonlinear,
we conclude that $\uinf$ must be constant on $H^+$ and on $H^-$, i.e.\
there exist $u^-, u^+\in\R$ such that
\begin{equation}\label{uinf}
\uinf = u^+\, \one_{H^+} + u^-\, \one_{H^-}
\end{equation}
compare  \cite[Proposition 7(b)]{DLOW}.
Indeed let $\bar z\in H^+$ be a Lebesgue point of $u^\infty$ and $-\|u\|_\infty -1<\bar k<u^\infty(\bar z)$ 
such that $\L^n(\{u^\infty =\bar k\})=0$.
Fix $\tau>0 $ and convolve with a nonnegative smooth kernel $\delta_\eps$ supported in $B_\eps$:  
for $\eps<\eps(\tau, \bar z)$ sufficiently small
\[
\delta_\eps * \chi(\bar k,u^\infty) (\bar z) \geq 1-\tau.
\]
Thanks to (GLN) we can choose  $n$ values $k_1,\dots,k_n$ (depending on $\tau$, $\eps$ and $\bar k$)  with $|k_n-\bar k|$ sufficiently small and  such that $\bar k<k_1<\dots<k_n$, $\{\a^+(k_i)\}$ are linearly independent and
\begin{equation}\label{11}
\delta_\eps * \chi(k_n,u^\infty) (\bar z) \geq 1-2\tau.
\end{equation}
For every $z$ the function $k\mapsto\chi(k,u^\infty(z))$ is decreasing, and so it remains when 
we convolve it with $\delta_\eps$: in particular
\begin{multline}\label{12}
\delta_\eps * \chi(\bar k,u^\infty)(z)\geq \delta_\eps * \chi(k_1,u^\infty)(z) \geq\dots\geq \delta_\eps * \chi(k_n,u^\infty)(z) \\
  \forall z\in H^+_\eps:=\{ z\cdot \nu(z_0)>\eps\}.
\end{multline}
Equation \eqref{f:transport}, which holds also for $\delta_\eps * \chi(k,u^\infty)$, implies that \(\delta_\eps * \chi(k_i,u^\infty)\) is constant along 
lines parallel to \(\a^+(k_i)\). Since the \(\{\a^+(k_i)\}\) are linearly independent, 
starting from \eqref{11} and exploiting \eqref{12} we obtain  $\rho_\eps * \chi(\bar k,u^\infty) \geq 1-2\tau$ in 
$H^+_\eps$.  Letting $\tau\downarrow 0$ we get
\[
\chi(\bar k,u^\infty)  \geq 1 \qquad \mbox{in }H^+. 
\]
Since ${\bar k}$ can be taken arbitrarily close to $u^\infty(\bar z)$,  $u^\infty$ is constantly equal to \(u^\infty(\bar z)\).
 A completely analogous argument holds for $H^-$. In particular \(\Gamma_{u,\mathcal N}(z_0)\ne \emptyset\).

\bigskip\noindent
\textsl{Step 5 (Characterization of traces)}. By \eqref{f:bkinetic} and the special form \eqref{uinf} of \(\uinf\),  we deduce that 
\begin{equation}
\label{f:trac}
\chi(k, u^+)\, \a^+(k)\cdot\nu(z_0)
- \chi(k, u^-)\, \a^-(k)\cdot\nu(z_0)
= -\partial_k  \big(h(z_0)\, \lambda_y(k)\big)
\qquad\text{in}\ \mathcal{D}'(\R).
\end{equation}
Let us now show as the above equality uniquely determines \(u^\pm\) whenever  $u^+\neq u^-$, in 
particular they do not depend on the choice of the subsequence \((r_j)\).  
To this end, let $(u_{\rho_j})$ be another converging subsequence of $(u_r)$: by Step 4 we have  
\begin{equation*}
u_{\rho_j} \to \vinf := v^+ \one_{H^+} + v^- \one_{H^-} \qquad\text{in}\ L^1(B_1),
\end{equation*}
so that the pair $(v^-, v^+)$ also satisfies~\eqref{f:trac}.
Subtracting the equation satisfied by the pair $(u^-, u^+)$ we get for almost every \(k\in \R\),
\[
[\chi(k, u^+) - \chi(k, v^+)] \, \a^+(k)\cdot\nu(z_0)
=
[\chi(k, u^-) - \chi(k, v^-)] \, \a^-(k)\cdot\nu(z_0)
\]
that is
\[
\sign(u^+ - v^+)\one_{(u^+, v^+)}(k) \, \a^+(k)\cdot\nu(z_0)
=
\sign(u^- - v^-)\one_{(u^-, v^-)}(k) \, \a^-(k)\cdot\nu(z_0).
\]
Since, again by the assumption (GNL) of genuine nonlinearity,
the functions 
\[
k\mapsto \a^{\pm}(k)\cdot\nu(z_0)
\]
cannot vanish
on any interval, 
the two intervals
$I(u^-, u^+)$ and $I(v^-,v^+)$ must coincide  \footnote{
Here, $I(a,b)$ denotes the interval $[a,b]$ if $a\leq b$ or the
interval $[b,a]$ if $b<a$.}.
If $u^-\neq u^+$, the condition $I(u^-,u^+) = I(v^-,v^+)$ can be satisfied
either in the case $v^- = u^-$, $v^+ = u^+$ or 
in the case $v^- = u^+$, $v^+ = u^-$.
On the other hand, this second possibility is excluded by the fact that the map
$r \mapsto u(y+rz)$
is continuous from $(0,1]$ to $L^1(B_1)$. Indeed, since 
\begin{gather*}
u_{r_j} \to \uinf = u^+ \one_{H^+} + u^- \one_{H^-},
\\
u_{\rho_j} \to \vinf = u^- \one_{H^+} + u^+ \one_{H^-},
\end{gather*}
we have 
\[
\int_{B_1} |u_{r_j} - \uinf| \to 0,
\qquad
\int_{B_1} |u_{\rho_j} - \uinf| \to \int_{B_1} |\vinf - \uinf| =: m\neq 0.
\]
By the continuity of the map
\[
(0,1] \ni r \mapsto \int_{B_1} u_{r}\,,
\]
and the relative compactness of the family \((u_r)_r\),  we can find a third sequence $(u_{s_j})$ such that
\begin{gather}
u_{s_j} \to w^{\infty} := w^+ \one_{H^+} + w^- \one_{H^-}\,,\notag \\
\int_{B_1} |w^{\infty} - \uinf| = \frac{m}{2} \leq \int_{B_1} |w^{\infty} - \vinf|,
\label{f:contuv}
\end{gather}
But then we must have $I(w^-,w^+) = I(u^-,u^+) = I(v^-, v^+)$, so that
either $w^- = u^-$ and $w^+=u^+$, or $w^- = u^+$ and $w^+ = u^-$,
and in each case we get a contradiction
with~\eqref{f:contuv}.

In conclusion, if $u^-\neq u^+$ then all subsequences of $(u_r)$ must converge
to the same function $\uinf$, hence the traces are uniquely determined.

\smallskip
In the case $u^- = u^+$, reasoning as above we can always conclude that 
$w^-=w^+$ for every $(w^-,w^+)\in \Gamma_{u,\mathcal N}$. Moreover exploiting again the continuity of
 the map $r \mapsto \int_{B_1} u_{r}$ we get that $\Gamma_{u,\mathcal N}$ is a compact connected set. 
Finally the Rankine--Hugoniot condition follows from Lemma~\ref{lemma:campo}, thus concluding the proof.
\end{proof}

The following Lemma has been used in the proof of Theorem \ref{thm:tracce}.

\begin{lemma}[Strong pre--compactness of blow-ups]
\label{l:panov}
The family $(u_r)$ defined in \eqref{rescal} is pre--compact in  $L^1(B_1)$.
\end{lemma}

\begin{proof}[Proof of Lemma \ref{l:panov}]
For every $r>0$, the function $u_r$ is a solution to
\[
\diver_z \A_r(z, u_r(z)) = 0,
\]
hence
\[
\diver_z \Abar_{z_0}(z, u_r(z)) = -\diver_z \left[\A_r(z,u_r(z)) - \Abar_{z_0}(u_r(z))\right].
\]
We claim that the family of functions
\[
q_r(z) := \A_r(z,u_r(z)) - \Abar_{z_0}(z,u_r(z))
\]
is pre--compact in $L^2(B_1)$, so that
$(\diver_z \Abar_{z_0}(z,u_r(z)))_r$ is pre--compact in the negative Sobolev space $W^{-1,2}(B_1)$.
If this condition is satisfied, then by
\cite[Thm.~6]{Pan1} we can conclude that 
$(u_r)$ is pre--compact in the strong $L^1(B_1)$ topology.

Let us consider the functions
\[
f_{r,z}(v) := |\A_r(z,v) - \Abar_{z_0}(z,v)|,
\qquad r>0,\ z\in B_1,\ v\in\R.
\]
By \eqref{conv1} 
 \begin{equation}\label{f:point}
\lim_{r\downarrow 0} f_{r,z}(v) = 0
\qquad \text{for every}\ z\in B_1\setminus D_0\ \text{and}\ \forall v\in\R,
\end{equation}
where $D_0\subset B_1$ is a set of Lebesgue measure $0$.
Moreover
\[
|f_{r,z}(v) - f_{r,z}(v')| \leq
|\A_r(z,v) - \A_r(z,v')| + |\Abar_{z_0}(z,v) - \Abar_{z_0}(z,v')|
\leq 2 \|\partial_v \A\|_{\infty} |v-v'|,
\]
hence $(f_{r,z})_r$ is an equi-Lipschitz family of functions converging 
 pointwise to $0$ for every  $z\in B_1\setminus D_0$.

Let $L = \|u\|_\infty$
and let $(v_k)\subset [-L, L]$ be a countable dense set in $[-L, L]$.
Using a diagonal argument, we can construct a sequence $(r_j)$ converging to $0$
such that
\[
\lim_{j\to+\infty} f_{r_j,z} (v_k) = 0
\qquad
\forall z\in B_1\setminus D,\
\forall k\in\N,
\]
where $D\supseteq D_0$ is a set of Lebesgue measure $0$.

Using the classical argument in the proof of the Ascoli--Arzel\`a compactness theorem,
we have that,
for every $z\in B_1\setminus D$,
the sequence $(f_{r_j,z})_j$ converges uniformly to $0$ in $[-L, L]$.
In other words, 
\[
g_j(z) := \sup_{|v|\leq L} | \A_{r_j}(z,v) - \Abar_{z_0}(z,v) | \to 0.
\qquad \forall z\in B_1\setminus D.
\]
Since the functions $g_j$ are equi-bounded, they converge to $0$ in $L^2(B_1)$.
Moreover,
\[
|q_{r_j}(z)|^2  := |\A_r(z,u_{r_j}(z)) - \Abar_{z_0}(z,u_{r_j}(z))|^2
\leq g_j(z)^2,
\]
so that the sequence $(q_{r_j})_j$ converges to $0$
in $L^2(B_1)$
and the claim is proved.
\end{proof}

\section{Proof of Theorem \ref{thm:representation}}\label{sec:representation}

In this section we prove Theorem \ref{thm:representation}. 
To this end we will need two technical lemmas: 
the first one is a slight generalization of classical arguments used in \cite{Kruz}. 
The second one allows to study the limiting behavior of the incremental quotient of \(\A\) in the spirit of  \cite[Thm.~2.4]{Ambrosio2004} and \cite[Lemma II.1]{DiPL} and it is crucial in the proof of 
Theorem~\ref{thm:representation}. 
For the sake of exposition we postpone the proofs of both lemmas  at the end of the section.

\begin{lemma}\label{lemma paraculo}
Let $f:\R^n\times\R\rightarrow \R^m$ satisfy the following assumptions:
\begin{itemize}
\item $z\mapsto \sup_v|f(z,v)|\in L^1_{\rm loc}(\R^n)$;
\item $|f(z,v) - f(z,v')|\leq g(z)\omega(|v-v'|)$ for some $g\in L^1_{\rm loc}$ and some modulus of continuity $\omega$. 
\end{itemize}
Then for every $u, v\in L^\infty_{\rm loc}(\R^n)$ 
\[
\begin{gathered}
|f(z+ \tau,u(z)) - f(z,u(z))| \rightarrow 0 
\\
\sign(u(z+\tau) - v(z)) [f(z+\tau, u(z+\tau)) - f(z,v(z))] \rightarrow \sign(u(z) - v(z)) [f(z, u(z)) - f(z,v(z))] 
\end{gathered}
\]
in $L^1_{\rm loc}$ as ${\tau\to 0}$.
\end{lemma}

\begin{lemma}[Uniform differential quotients]
	\label{lemma:luigi}
	Let $\A$ satisfy (H1)--(H5) and let $w\in\R^n$. 
	Then there exists a measurable set $D = D_w\subset\R^n$, with $\L^n(D) = 0$, such that
	the difference quotients for $\A$ can be canonically written as
	\[
	\frac{\A(z+\eps w, v) - \A(z,v)}{\eps}
	= \A^1_\eps(z,v) + \A^2_\eps(z,v)
	\]
	where $\A^1_\eps$ and $\A^2_\eps$ satisfy the following properties:
	\begin{itemize}
		\item[(i)]
		$\displaystyle
		\lim_{\eps\downarrow 0} \A^1_\eps(z,v) =
		\nabla_z \A(z, v)\cdot w,
		\qquad \forall v\in\R\ \text{and}\ z\in\R^n\setminus D;$
		\item[(ii)]
		The family of functions $h_\eps\colon\R^n\to\R$ defined by
		\[
		h_\eps(z) := |w|\sup_{v\in\R} \left| \A^1_\eps(z,v)\right|
		\]
		is equi-integrable;
		\item[(iii)]
		For every compact set $K\subset\R^n$ we have
		\[
		\int_K \sup_{v\in\R} \left|\A^2_\eps(z,v)\right|\, dz \leq
		\sigma^s(K_{\eps |w|}) |w|,
		\]
		where $K_{\tau} := K + B_\tau(0)$.
	\end{itemize}
\end{lemma}

\begin{proof}[Proof Theorem \ref{thm:representation}]
We divide the proof into several steps:

\medskip
\noindent
{\em Step 1: Doubling of variables}. We follow the classical technique of Kruzhkov \cite{Kruz}.  
Let  \(u(z)\) and \(v(z')\) be WES: let us set \(k=v(z')\) in  \eqref{eq:ent} for \(u\) and   \(k=u(z)\) in \eqref{eq:ent} for \(v\). 
Let us also choose a 
 test 
function \(\Phi(z,z')=\varphi(z+z')\delta_\eps(z-z')\) where \(\varphi\in C_c^1(\R^n)\) is nonnegative and  \(\delta_\eps\) 
is the usual smooth  approximation of the identity in \(0\):
\[
\delta_\eps(\zeta)=\frac{1}{\eps^n} \psi\big(\zeta/\eps\big)\,\qquad \psi \in C_c^1(B_1),\quad \int\psi=1,\quad \psi(z) = \psi(-z).
\]
Multiplying both equations by $\Phi$, integrating in \(z\) and \(z'\) and subtracting the corresponding inequalities we obtain
\[
\begin{split}
\iint &\Big\{\delta_\eps(z-z')\nabla \vphi(z+z')+\vphi(z+z')\nabla \delta_\eps(z-z')\Big\}\sign\big(u(z)-v(z')\big)\big(\A(z,u(z))-\A(z,v(z'))\big) \\
-&\sign\big(u(z)-v(z')\big)\Div^a_z\A(z,v(z'))\vphi(z+z')\delta_\eps(z-z')\\
+&\Big\{\delta_\eps(z-z')\nabla \vphi(z+z')-\vphi(z+z')\nabla \delta_\eps(z-z')\Big\}\sign\big(v(z')-u(z)\big)\big(\A(z',v(z'))-\A(z',u(z))\big)\\
-&\sign\big(v(z')-u(z)\big)\Div^a_{z'}\A(z',u(z))\vphi(z+z')\delta_\eps(z-z')  dzdz' \ge- 2\iint \delta_\eps(z-z')\vphi(z+z')dz'd\mu(z).
\end{split}
\]
This can be written as 
\begin{equation}
\label{kru1}
I^\eps_1-I^\eps_2+I^\eps_3\ge- 2\iint \delta_\eps(z-z')\vphi(z+z')dz'd\mu(z)
\end{equation}
where
\[
\begin{split}
I^\eps_1= \iint & \psi(w)\nabla \vphi(2z-\eps w) \sign\big((u(z)-v(z -\eps w)\big)\\
&\times\Big\{\A(z,u(z))+\A(z - \eps w ,u(z))-\A(z,v(z - \eps w))-\A(z- \eps w,v(z- \eps w)) \Big\}\,  dw dz   ,
\\
%
I_2^\eps=\iint &  \vphi(2z-\eps w)\sign\big(u(z)-v(z-\eps w)\big)\\
&\times \Big\{\nabla \psi(w)\frac{\A(z-\eps w,u(z))-\A(z,u(z))}{\eps}-\psi(w)\Div^a_{z}\A(z-\eps w,u(z))\Big\}\,dwdz,\\
I_3^\eps=\iint & \vphi(2z+\eps w)\sign\big(u(z+\eps w)-v(z)\big)\\
&\times \Big\{\nabla \psi(w)\frac{\A(z,v(z))-\A(z+\eps w,v(z))}{\eps}-\psi(w)\Div^a_{z}\A(z+\eps w,v(z))\Big\}  dwdz.
\end{split}
\]
Regarding $I^\eps_1$,  Lemma \ref{lemma paraculo} implies that 
\begin{equation}
\label{kru2}
I^\eps_1\to 2 \int \nabla \varphi(2z)\sign\big((u(z)-v(z)\big) \,  dz\\
\big(\A(z,u(z))-\A(z,v(z))\big).
\end{equation}
We will  now show that
\begin{equation}
\label{eq:incremento}
\limsup_{\eps \to 0} |I_2^\eps-I_3^\eps|\le C \|\varphi\|_{\infty} |\sigma^s|(\spt \vphi).
\end{equation}
This, together with \eqref{kru1} and \eqref{kru2},  will then give that, in the sense of distributions,
\begin{equation}\label{eq:Katouv}
\Div_{z} \Big(\sign\big((u(z)-v(z)\big)\\
\big(\A(z,u(z))-\A(z,v(z))\big)\Big)\le 2\mu + C |\sigma^s|=: \beta
\end{equation}
where \((\mu + C |\sigma^s|)(\R^n\setminus \mathcal N)=0\).
In turn the left hand side of \eqref{eq:Katouv} is a signed measure, 
which we denote by $\alpha$, for which:
\[
\alpha\leq \alpha^+ = \alpha^+\res\mathcal N = (\alpha \res \mathcal N)^+ \leq \beta. 
\]
Since the map \[
(u,v)\mapsto \sign(u-v\big)\big(\A(z,u)-\A(z,v)\big),
\]
is Lipschitz and 
\[
\A(z_0+\eps z ,v)\to \A^\pm(z_0,v)\qquad\textrm{in \(L^1_{\rm loc}\) for every \(v\in \R\)},
\]
by arguing as in Lemma \ref{lemma paraculo}  the traces of the vector field 
\[
z\mapsto \sign\big((u(z)-v(z)\big)\big(\A(z,u(z))-\A(z,v(z))\big)
\]
 exist for \(\mathcal H^{n-1}\) almost every \(z\in \mathcal N\)  and  are given by 
\[
\sign\big((u^\pm(z)-v^\pm(z)\big)\big(\A^\pm(z,u^\pm(z))-\A^\pm(z,v^\pm(z))\big).
\]
A direct application of  Lemma \ref{lemma:campo} yields the desired representation \eqref{W}.

To show uniqueness of the traces at points where  $w(z)\neq 0$  we note that we only have to discuss  the case when (say) \(v^-(z)=v^+(z)=v\) and \(u^-(z)\ne u^+(z)\), otherwise either the traces are unique by Theorem \ref{thm:tracce} or \(w=0\). The Rankine--Hugoniot condition gives
\begin{equation}\label{sticazzi}
w(z) = [ \sign (u^+(z) - v) - \sign (u^-(z) - v)  ] [ \A^+(z,u^+(z)) - \A^+(z,v) ]\cdot \nu(z).
\end{equation}
Moreover we know by  Theorem \ref{thm:tracce} that if \(\Gamma_{v,\mathcal N}\) is not a singleton it contains pairs \((v',v')\) with \(v'\) ranging in a non trivial interval \([a,b]\). Being $w$  uniquely determined and non zero we have that~\eqref{sticazzi} holds for any such \(v'\in [a,b]\) and that \(v'\in I(u^-,u^+)\). 
This implies that 
\begin{equation*}
 [ \A^+(z,v') - \A^+(z,v) ]\cdot \nu(z) = 0\qquad \forall\, v,v'\in [a,b]\,,
\end{equation*}
contradicting the genuine nonlinearity  assumption (GLN). 

\smallskip

In order to conclude the proof of the Theorem we only have to show the validity of  \eqref{eq:incremento}.  
According to Lemma~\ref{lemma:luigi} above we can write
\[
\frac{\A(z-\eps w,u(z))-\A(z,u(z))}{\eps}=\A_{\eps,w}^1(z)+\A^2_{\eps,w}(z)\,\qquad
\]
where
\[
\A_{\eps,w}^1(z)\stackrel{L^1_{\rm loc}}{\longrightarrow}- \nabla \A(z,u(z)) \cdot w
\]
and 
\[
\int dz |\A^2_{\eps,w}(z) |\varphi(z)\le  |w| \|\varphi\|_{\infty} |\sigma^s|((\spt \vphi)_{\eps |w|})\,.
\]
Hence, by also using Lemma \ref{lemma paraculo}, we obtain that 
 \[
 \begin{split}
 I_2^\eps = {} & \iint  \vphi(2z-\eps w)\sign\big(u(z)-v(z-\eps w)\big)
 \\ & \times \big\{-\nabla \psi(w)\nabla \A(z,u(z)) \cdot w-\psi(w)\Div^a_{z}\A(z,u(z))\big\}\,dwdz
+ R^\eps_{1}+R^\eps_{2}\,,
 \end{split}
 \]
where
\[
\limsup_{\eps \to 0} |R_{1}^\eps|\le C(\psi) \|\varphi\|_{\infty} |\sigma^s|(\spt \vphi)\quad\textrm{and}\quad \lim_{\eps\to 0} |R_{2}^{\eps}|=0.
\]
By applying the same decomposition to \(I_{3}^\eps\) we obtain, after a change of variable,  that 
\[
\begin{split}
\limsup_{\eps\to 0} |I_2^\eps-I_3^\eps|&\le 2C(\psi) \|\varphi\|_{\infty} |\sigma^s|(\spt \vphi)
\\
&+\limsup_{\eps \to 0}\Big|\iint  \vphi(2z-\eps w)\sign\big(u(z)-v(z-\eps w)\big)\\
&\times \Big\{\nabla \psi(w)\big[\nabla \A(z,u(z))\cdot w-\nabla \A(z-\eps w,v(z-\eps w))\cdot w\big] \\
&\quad+\psi(w)\big[\Div^a_{z}\A(z,u(z))-\Div^a_{z}\A(z-\eps w,v(z-\eps w))\big]\Big\}dwdz
\Big|.
\end{split}
\]
By Lemma \ref{lemma paraculo} the latter integral converges to 
\[
\begin{split}
&\iint  \vphi(2z)\sign\big(u(z)-v(z)\big)\big\{\nabla \psi(w)\nabla \A(z,u(z))\cdot w+\psi(w)\Div^a_{z}\A(z,u(z))\big\}dwdz
\\
&-\iint  \vphi(2z)\sign\big(u(z)-v(z)\big)\big\{\nabla \psi(w)\nabla \A(z,v(z))\cdot w+\psi(w)\Div^a_{z}\A(z,v(z))\big\}dwdz.
\end{split}
\]
Integrating by parts with respect to the \(w\) variable, we get that both integrals are zero, thus concluding the proof of \eqref{eq:incremento}.
\end{proof}

We conclude the Section by proving Lemma \ref{lemma paraculo} and Lemma  \ref{lemma:luigi}.

\begin{proof}[Proof of Lemma \ref{lemma paraculo}]
Let  $Q\subset\R$ be a countable dense set:  by the continuity of translations in $L^1$  
\[
|f(z+ \tau,u) - f(z,u)| \rightarrow 0 \qquad\textrm{in $L^1_{\rm loc}$ for every  $u\in Q$}. 
\]
If now $u\in L^\infty$, there exists $u_k = \sum_{i=1}^{N_k} u^i_k\one_{A^i_k}$ with 
$u^i_k\in Q$ and such that $\|u-u_k\|_{\infty} \to 0$. Hence for every compact set $K\subset\R^n$
\[
\begin{split}
\int_K|f(z+ \tau,u(z)) - f(z,u(z))|dz &\leq \omega(\|u-u_k\|_{\infty}) \int_K (g(z) + g(z+\tau)) dz \\
& + \int_K|f(z+ \tau,u_k(z)) - f(z,u_k(z))|dz  \\
 &\leq o_k(1) +  \sum_{i=1}^{N_k} \int_{K\cap A^i_k}| f(z + \tau, u^i_k) -  f(z , u^i_k)|dz 
\end{split}
\]
where $o_k(1) \to 0$ independently on $\tau$ as $k\to\infty$. Passing to the limit first on $\tau$ and then on $k$ proves the first claim. 
To prove the second claim note that thanks to what we have proved it is enough to show that 
\[
\sign(u(z+\tau) - v(z)) [f(z, u(z+\tau)) - f(z,v(z))] \rightarrow \sign(u(z) - v(z)) [f(z, u(z)) - f(z,v(z))] 
\]
in $L^1_{\rm loc}$ as ${\tau\to 0}$. Since the map
\[
(u,v) \mapsto \sign(u- v) [f(z, u) - f(z,v)] 
\] 
has  modulus of  continuity $2\omega$  independently on $z$ this plainly follows by the continuity of translations in \(L^1\).
\end{proof}

\begin{proof}[Proof of Lemma \ref{lemma:luigi}]
	Up to dilating and rotating we can assume that $w=e_n$. We will write
	$z=(z', z_n)$ with $z'\in\R^{n-1}$ and $z_n\in\R$.
	
       Let $Q=(v_j)\subset\R$ be a countable dense set in $\R$. By slicing theory for \(BV\) functions, see \cite[Chapter 3]{AFP}, 
	for every $j\in\N$ there exists a set $D_j\subset\R^n$
	with $\L^n(D_j) = 0$, such that,
	for every $z\in\R^n\setminus D_j$, the function
	$t \mapsto \A(z', z_n+t, v_j)$ belongs to $BV(\R)$ and 
	the absolutely continuous part of its derivative,
	denoted by $\frac{\partial\A}{\partial z_n}(z', z_n+t, v_j)$,
	coincides with $\nabla_z \A(z', z_n+t, v_j)\cdot e_n$. Hence for  $j\in\N$ and $z\in \R^n\setminus D_j$
	we define
	\[
	\A^1_\eps(z',z_n, v_j) = \int_0^1 \frac{\partial\A}{\partial z_n}(z', z_n + \eps t, v_j)\, dt= \int_0^1 \nabla_z \A(z', z_n+\eps t, v_j)\cdot e_n\, dt\,.
	\]
	From \cite[Thm.~2.4]{Amb} there exists a measurable set $D\subset\R^n$, with
	$D\supset \mathcal{C}_{\A} \cup \bigcup_j D_j$ and $\L^n(D) = 0$,
	such that
	\begin{equation}\label{f:amb}
	\lim_{\eps\downarrow 0} \A^1_\eps(z,v_j) =
	\nabla_z \A(z, v_j)\cdot e_n,
	\qquad \forall j\in\N\ \text{and}\ z\in\R^n\setminus D.
	\end{equation}
	Moreover, up to add to \(D\) a set of Lebesgue measure zero, we can assume that  every \(z\) in \(\R^n\setminus D\)
	 is a Lebesgue point for the function \(g\) appearing in  (H4) and that 
        \begin{equation}\label{geps}
	G_\eps(z) := \int_0^1 g(z', z_n + \eps t)\, dt\,\to g(z)\qquad \textrm{as $\eps\downarrow 0$.}
	\end{equation}
	Let us now fix $z\in\R^n\setminus D$ and  $j,k\in\N$: by (H4) we have that
	\begin{equation}\label{f:cauchy}
	\begin{split}
	\left|\A^1_\eps(z, v_j) - \A^1_\eps(z,v_k)\right|
	& \leq
	\int_0^1\left|\nabla_z\A(z', z_n+\eps t, v_j) - \nabla_z \A(z', z_n + \eps t, v_k)\right|\, dt
	\\ &
	\leq G_\eps(z) \, \omega(|v_j - v_k|).
	\end{split}
	\end{equation}
       Let us now take \(v\in \R\) and  $v_{j}\in Q$ with \(v_j\to v\). By~\eqref{f:cauchy} 
	$\left(\A^1_\eps(z, v_{j})\right)_j$ is a Cauchy sequence,
	hence it converges to  a unique limit  $\ell_\eps(z,v)$. Let us define for \(v\in \R\) and \(z\in \R^n\setminus D\)
	\[
	\A^1_\eps(z, v)=\ell_\eps(z,v)
	\]
	and 
	\[
	\A^2_\eps(z, v)=\frac{\A(z+\eps w, v) - \A(z,v)}{\eps}-\A^1_{\eps}(z,v).
	\]
	We now verify the validity of (i)-(iii). First of all  \eqref{f:cauchy} implies
	\begin{equation}\label{f:cauchy2}
	\left|\A^1_\eps(z, v) - \A^1_\eps(z,v')\right|
	\leq G_\eps(z) \, \omega(|v - v'|)
	\qquad\forall v,v'\in\R.
	\end{equation}
	Moreover, according to \cite[Lemma 3.4]{ACDD}, we can add to \(D\) a set of measure zero outside which \(\nabla \A(z,v)\) is well defined and continuous in \(v\). Hence for \(z\in \R^n\setminus D\) and \(v\in \R\), by \eqref{f:cauchy2} and (H4), we have 
	\[
	\begin{split}
	\left|\A^1_\eps(z,v) - \nabla_z\A(z,v)\cdot e_n\right|
	\leq {} &
	\left|\A^1_\eps(z,v) - \A^1_\eps(z,v_j)\right|
	+
	\left|\A^1_\eps(z,v_j) - \nabla_z\A(z, v_j)\cdot e_n\right|
	\\ & +
	\left|\nabla_z\A(z,v_j)\cdot e_n - \nabla_z \A(z,v)\cdot e_n\right|
	\\ \leq {} &
	G_\eps(z)\, \omega(|v-v_j|) 
	+
	\left|\A^1_\eps(z,v_j) - \nabla_z\A(z, v_j)\cdot e_n\right|
	\\ & +
	g(z)\, \omega(|v-v_j|)\,.
	\end{split}
	\]
	Taking the limsup as $\eps\downarrow 0$ 
	and taking into account~\eqref{f:amb} and \eqref{geps}
	we get
	\[
	\limsup_{\eps\downarrow 0}
	\left|\A^1_\eps(z,v) - \nabla_z\A(z,v)\cdot e_n\right|
	\leq
	2 g(z) \, \omega(|v-v_j|).
	\]
   Since $(v_j)$ in dense in $\R$, we conclude that~(i) holds.
	
	\medskip
	Let us prove (ii).
	For almost every \(z\in \R^n\) we have 
	\[
	h_\eps(z) = \sup_{j\in\N} \left|\A^1_\eps(z, v_j)\right|
	\leq
	\int_0^1 \sigma^a(z', z_n + \eps t)\, dt.
	\]
	Since $\sigma^a\in L^1(\R^n)$  there exists a superlinear, convex,
	increasing function $\psi\colon [0,+\infty)\to [0, +\infty)$ such that
	\[
	\int_{\R^n}\psi(\sigma^a(z))\, dz < +\infty.
	\]
	Then, by Jensen's inequality,
	\[
	\begin{split}
	\int_{\R^n} \psi(h_\eps(z))\, dz
	& \leq
	\int_{\R^{n-1}} \int_{\R} \psi\left(\int_0^1 \sigma^a(z', z_n+\eps t)\, dt\right)\, dz_n\, dz'
	\\ & \leq
	\int_0^1 \int_{\R} \int_{\R^{n-1}} \psi(\sigma^a(z', z_n + \eps t)) dz'\, dz_n\, dt
	\\ & = 
	\int_{\R^{n}} \psi(\sigma^a(z))\, dz < +\infty\,.
	\end{split}
	\]
	By the Dunford--Pettis compactness criterion we conclude that
	the family $(h_\eps)$ is equi--integrable in $L^1(\R^n)$,
	hence~(ii) is proved.
	
	\medskip

	To  conclude (iii), for every $j\in\N$ we let $\varphi^j_{z'}(z_n) := \A(z',z_n, v_j)$ and we note that for almost every $z_n$
	\[
	\begin{split}
	\frac{\A(z+\eps e_n, v_j) - \A(z,v_j)}{\eps}
	& =\frac{1}{\eps}\, D\varphi^j_{z'} ([z_n,z_n+\eps])
	\\ & = \int_0^1 \nabla_z\A(z+t \eps e_n, v_j) \cdot e_n\, dt
	+\frac{1}{\eps}\, D^s\varphi^j_{z'}([z_n,z_n+\eps]),
	\end{split}
	\]
	hence
	\[
	\left|\A^2_{\eps}(z,v_j)\right| \leq\frac{1}{\eps}\, |D^s\varphi^j_{z'}|([z_n,z_n+\eps]).
	\]
	If $K$ is a compact subset of $\R^n$ we get
	\begin{equation}
	\label{f:stima2j}
	\begin{split}
	\int_K \left|\A^2_{\eps}(z,v_j)\right|\, dz
	& \leq
	\int_{\R^{n-1}} dz' \int_{\{z_n:\ (z',z_n)\in K\}} dz_n
	\frac{1}{\eps}\, |D^s\varphi^j_{z'}|([z_n,z_n+\eps])
	\\ & =
	\int_{\R^{n-1}} dz' \int_{\{z_n:\ (z',z_n)\in K\}} dz_n
	\int_{\R}  \frac{1}{\eps}\,\one_{[z_n,z_n+\eps]}(t) \, d|D^s\varphi^j_{z'}|(t)
	\\ & \leq
	\int_{\R^{n-1}} dz' \int_{\{t:\ (z',t)\in K_\eps\}} 
	d|D^s\varphi^j_{z'}|(t)
	\int_{\R} dz_n \frac{1}{\eps}\,\one_{[z_n,z_n+\eps]}(t) \, 
	\\ & \leq
	|D^s \A(\cdot, v_j)|(K_\eps) \leq
	\sigma^s(K_\eps).
	\end{split}
	\end{equation}
	Let now $v\in\R$.
	From \eqref{f:cauchy2},  \eqref{f:stima2j} and (H1) we get 
	\[
	\begin{split}
	\int_K \left|\A^2_{\eps}(z,v)\right|\, dz
	\leq {} &
	\int_K \left|\A^2_{\eps}(z,v)-\A^2_{\eps}(z,v_j)\right|\, dz
	+ 
	\int_K \left|\A^2_{\eps}(z,v_j)\right|\, dz
	\\ \leq {} &
	\int_K\left|\frac{\A(z+\eps w, v) - \A(z,v)}{\eps}
	-\frac{\A(z+\eps w, v_j) - \A(z,v_j)}{\eps}\right|\, dz
	\\ & + 
	\int_K \left|\A^1_{\eps}(z,v)-\A^1_{\eps}(z,v_j)\right|\, dz
	+
	\int_K \left|\A^2_{\eps}(z,v_j)\right|\, dz
	\\ \leq {} &
	\frac{2}{\eps}\, M\,
	\L^n(K_\eps)\, |v-v_j|
	+\left(\int_{K} G_\eps(z)\, dz\right) \, \omega(|v-v_j|)
	+ \sigma^s(K_\eps).
	\end{split}
	\]
	Exploiting the density of  $(v_j)$, we
	get (iii).
\end{proof}

\section{Proofs of Theorems \ref{thm:gentropy} and \ref{thmW110}}\label{sec:fine}

In this Section we briefly sketch the proofs of Theorems \ref{thm:gentropy} and \ref{thmW110}.

\begin{proof}[Proof of Theorem \ref{thm:gentropy}]
By Theorem \ref{thm:representation} we have that any two \(\boldsymbol{\mathcal G}\)-entropy solutions \(u,v\) satisfy the Generalized Kato Inequality \eqref{Kato2}. By the usual test function argument, see \cite{Kruz} we then obtain that 
 for every \(T>0\) and every \(R>0\)
\begin{equation*}
\begin{split}
& \int_{B_R}\ 
|u(T,x)-v(T,x)| \, dx
\\ & \leq 
\int_{B_{R+VT}}\ 
|u(0,x)-v(0,x)| \, dx
+ \int_{{\mathcal N}\cap ([0,T]\times B_{R+VT})} w(t,x)\,d\H^{N}(t,x) 
\,,
\end{split}
\end{equation*}
where \(w(t,x)\) is given by \eqref{W}. Since \(u,v\) are  \(\boldsymbol{\mathcal G}\)-entropy solutions \(w\le 0\), from which uniqueness immediately follows. 
\end{proof}

\begin{proof}[Proof of Theorm \ref{thmW110}] If \(\F(\cdot, u)\) belongs to \(W^{1,1}\) it easily follows from the definition of supremum of measures that \(\sigma^s=0\) which implies that  \(\mathcal H^N(\mathcal N)=0\). Theorem \ref{thm:representation} then gives that any two entropy solutions satisfy a true Kato inequality:
\[
\partial_t|u-v|+\diver_x \Big(\sign (u-v) [\F(t,x,u)-\F(t,x,v)]\Big)\leq 0,
\]
from which the validity of the \(L^1\) contraction inequality is then straightforward.
\end{proof}

\def\cprime{$'$}
\providecommand{\bysame}{\leavevmode\hbox to3em{\hrulefill}\thinspace}
\providecommand{\MR}{\relax\ifhmode\unskip\space\fi MR }
\providecommand{\MRhref}[2]{%
	\href{http://www.ams.org/mathscinet-getitem?mr=#1}{#2}
}
\providecommand{\href}[2]{#2}

\end{document}